\documentclass[11pt]{article}

\usepackage{fullpage}

\usepackage{amssymb,amsmath,amsthm,latexsym}
\usepackage{mathtools}
\usepackage{bm}
\usepackage{bbm}
\usepackage{xspace}

\usepackage{url}
\usepackage{xcolor}
\usepackage{graphicx}
\usepackage{subfig}
\usepackage{wrapfig}
\usepackage{tikz-cd}

\usepackage{algorithm}
\usepackage{algorithmic}

\usepackage{tabularx}
\usepackage{array}
\usepackage{makecell}
\usepackage{multirow}
\usepackage{caption}
\usepackage{bookmark}
\usepackage{nicefrac}
\usepackage{booktabs}
\usepackage{siunitx}

\usepackage[capitalise]{cleveref}
\usepackage{hyperref}
\definecolor{niceblue}{rgb}{0.0,0.19,0.56}
\hypersetup{
  colorlinks = true,
  linkcolor  = blue,
  citecolor  = niceblue,
  urlcolor   = blue
}

\usepackage{natbib}

\usepackage[most]{tcolorbox}
\usepackage{pifont}

\definecolor{bgcolor}{rgb}{0.8,1,1}
\definecolor{bgcolor2}{rgb}{0.8,1,0.8}
\definecolor{PineGreen}{RGB}{0,110,51}
\definecolor{BrickRed}{RGB}{143,20,2}

\newcolumntype{Y}{>{\centering\arraybackslash}X}

\usepackage{mdframed} 
\usepackage{thmtools}

\definecolor{shadecolor}{gray}{0.9}
\declaretheoremstyle[
headfont=\normalfont\bfseries,
notefont=\mdseries, notebraces={(}{)},
bodyfont=\normalfont,
postheadspace=0.5em,
spaceabove=1pt,
mdframed={
 skipabove=8pt,
 skipbelow=8pt,
 hidealllines=true,
 backgroundcolor={shadecolor},
 innerleftmargin=4pt,
 innerrightmargin=4pt}
]{shaded}

\declaretheorem[style=shaded,within=section]{definition}
\declaretheorem[style=shaded,sibling=definition]{theorem}
\declaretheorem[style=shaded,sibling=definition]{proposition}
\declaretheorem[style=shaded,sibling=definition]{assumption}
\declaretheorem[style=shaded,sibling=definition]{corollary}

\declaretheorem[style=shaded,sibling=definition]{lemma}
\declaretheorem[style=shaded,sibling=definition]{remark}

\newcommand*{\R}{{\mathbb R}}

\DeclareMathOperator*{\argmin}{argmin}

\renewcommand{\|}{\parallel}

\newcommand{\N}{\mathbb{N}}

\newtcolorbox{remarkbox}{
    colback=gray!20,    
    colframe=black,     
    boxrule=0.5mm,      
    sharp corners,      
    left=2mm,           
    right=2mm,          
    top=2mm,            
    bottom=2mm          
}

\definecolor{PineGreen}{RGB}{0,110,51}
\definecolor{BrickRed}{RGB}{143,20,2}
\definecolor{mycolor}{RGB}{0,150,70}

\usepackage[colorinlistoftodos,bordercolor=orange,backgroundcolor=orange!20,linecolor=orange,textsize=scriptsize]{todonotes}

\newcommand{\cO}{{\cal O}}

\graphicspath{{plots/}}

\usepackage{accents}
\newlength{\dhatheight}

\usepackage{pgfplotstable} 
\pgfplotsset{compat=1.18}
\usetikzlibrary{automata, positioning, arrows, shapes, fit, calc, intersections}
\usepgfplotslibrary{statistics}
\IfFileExists{figure.tex}{\include{figure}}{}

\newcommand\swapifbranches[3]{#1{#3}{#2}}
\makeatletter
\MHInternalSyntaxOn
\patchcmd{\DeclarePairedDelimiter}{\@ifstar}{\swapifbranches\@ifstar}{}{}
\MHInternalSyntaxOff
\makeatother

\DeclarePairedDelimiterX{\inp}[2]{\langle}{\rangle}{#1, #2}
\DeclarePairedDelimiterX{\roundup}[1]{\lceil}{\rceil}{#1}
\DeclarePairedDelimiterX{\rbr}[1]{(}{)}{#1} 

\makeatletter
\newcommand{\mytag}[1]{%
  \refstepcounter{equation}%
  \edef\@currentlabel{\theequation}%
  {({\@currentlabel})}%
  \@bsphack
  \begingroup
    \@onelevel@sanitize\@currentlabelname
    \edef\@currentlabelname{%
      \expandafter\strip@period\@currentlabelname\relax.\relax\@@@%
    }%
    \protected@write\@auxout{}{%
      \string\newlabel{#1}{%
        {\@currentlabel}%
        {\thepage}%
        {\@currentlabelname}%
        {\@currentHref}{}%
      }%
    }%
  \endgroup
  \@esphack
}
\makeatother

\title{\bf Last Iterate Convergence of AdaGrad-Norm\\ for Convex Non-Smooth Optimization}
\author{\begin{tabular}{c}
     {\bf Margarita Preobrazhenskaia \quad
Makar Sidorov \quad
Igor Preobrazhenskii}\\
     {\small Yaroslavl Demidov State University, Russia}
\end{tabular}\\
\begin{tabular}{c}
     {\bf Eduard Gorbunov}\\
     {\small Mohamed Bin Zayed University of Artificial Intelligence, UAE}
\end{tabular}}

\begin{document}

\maketitle

\begin{abstract}
We study the convergence of the last iterate (i.e., the $(N+1)$-th iterate) of the AdaGrad method. Although AdaGrad --- an adaptive subgradient method --- underpins a wide class of algorithms, most existing convergence analyses focus on averaged (or best) iterates. We derive worst-case upper bounds on the suboptimality of the final point and show that, with an optimally tuned stepsize parameter, the last iterate converges at the rate $O(\nicefrac{1}{N^{\nicefrac{1}{4}}})$. We complement this guarantee with matching lower-bound constructions, proving that this rate is tight and that AdaGrad's last-iterate rate is strictly worse than the classical $O(\nicefrac{1}{N^{\nicefrac{1}{2}}})$ rate for its averaged iterate. Technically, our analysis introduces an exponent parameter that captures the growth of the cumulative squared subgradients; combined with the last-iterate inequality of \citet{zamani2025exact}, this reduces the problem to bounding a particular series.
\end{abstract}

\section{Introduction} 

First-order optimization methods are popular for solving large-scale optimization problems arising in various applications, including (but not limited to) Machine Learning \citep{bottou2018optimization}, Signal Processing \citep{eldar2012compressed}, Image Processing \citep{chambolle2011first}, and Optimal Control \citep{bryson1975applied}. The basic example of such a method is a subgradient method \citep{shor1985minimization}, and more recent variants including stochastic \citep{nemirovski2009robust}, adaptive \citep{mcmahan2010adaptive,duchi2011adaptive,kingma2015adam}, and parameter-free \citep{orabona2016training} extensions.

The convergence theory of these methods is well developed across a range of settings, including convex non-smooth optimization, which is the main focus of our paper. However, most available results in this area are stated for \emph{averaged} or \emph{best} iterates. By contrast, last-iterate convergence --- despite being the default choice in many implementations --- has been analyzed only in a small number of works \citep{shamir2013stochastic,jain2021making,zamani2025exact,kornowski2025gradient,zamani2024exact,defazio2023optimal,liu2025online,liu2023revisiting}. In particular, a tight worst-case characterization for the last iterate of the standard AdaGrad method \emph{without stepsize modifications} has been missing.

\paragraph{Contributions.} We analyze AdaGrad with the standard scalar stepsizes and a horizon-dependent scaling of the base stepsize. For convex problems with bounded subgradients, we establish a worst-case convergence bound of order $O(\nicefrac{1}{N^{\nicefrac{1}{4}}})$ for the last iterate under an optimal choice of the stepsize parameter (see Theorem~\ref{th_main} and Corollary~\ref{cor:main}). Our proof takes an asymptotic viewpoint and introduces a key parameter $\delta$ that captures the growth exponent of the cumulative squared subgradients over $N$ iterations. Combined with a last-iterate inequality for subgradient methods due to \citet{zamani2025exact}, this reduces the analysis to estimating the sum
$$\sum_{k=1}^{N-1}\frac{y_k}{(N-k)(1+\sum_{i=1}^ky_i)}$$ 
for $0<y_i\leqslant 1$, $1+\sum_{i=1}^{N-1}y_i= N^{2\delta}$, $\delta\in[0,\frac12]$. Moreover, we provide matching lower-bound constructions (see Theorem~\ref{thm:lower_bound}), showing that the $O(\nicefrac{1}{N^{\nicefrac{1}{4}}})$ rate \emph{cannot be improved} for the last iterate of standard AdaGrad; in particular, it is strictly slower than $O(\nicefrac{1}{N^{\nicefrac{1}{2}}})$ averaged-iterate rate for AdaGrad \citep{duchi2011adaptive}.

\paragraph{Structure of the paper.} Section~\ref{sec_prelim} contains definitions, assumptions and the problem statement. In Section~\ref{sec_related}, we discuss closely related works. Section~\ref{sec_main} contains the formulation and discussion of the main results. Section~\ref{sec_exper} describes the numerical experiments. Section~\ref{ref_concl} presents the conclusions and describes further research directions. The detailed proofs of the main results are deferred to Appendices~\ref{sec_app_A} and \ref{appendix:lower_bound}.

\section{Preliminaries}\label{sec_prelim} 

\paragraph{Notation.} Let $\langle\cdot,\cdot\rangle$ be the standard inner product in $\mathbb{R}^d$, and let $\|\cdot\|$ be the induced Euclidean norm. The subdifferential of a function $f:\R^d \to \R$ at a point $x$ is denoted by $\partial f (x) := \{g\in\mathbb{R}^d:\ \forall y \in\mathbb{R}^d\  f(y)\geqslant f(x)+\langle g,y-x\rangle\}$. We also use $P_X(y)$ to denote the Euclidean projection onto a closed convex set $X \subseteq \R^d$: $P_X(y):=\argmin_{x\in X}\|x-y\|$.

\paragraph{Problem setup.} We consider a classical convex optimization problem
\begin{equation}\label{eq:problem}
    \min_{x\in X}f(x),
\end{equation}
where $X\subset\mathbb{R}^d$ is a closed convex set, $f:X\to\mathbb{R}$ is the objective function. We make the following standard assumptions.

\begin{assumption}[Convexity and existence of the minimizer]\label{as:convexity}
    Function $f:X \to \R$ is convex and $x^* \in \argmin_{x\in X}f(x)$ exists.
\end{assumption}
\begin{assumption}[Bounded subgradients]\label{as:bounded_subgrad}
    Function $f:X \to \R$ has $G$-bounded subgradients: $\forall x, g$ such that $g\in \partial f(x)$ we have $\|g\| \leq G$.
    \begin{equation}
    \label{neq_G}
    \|g\|\leqslant G,\quad \forall g \in \partial f(x), x\in X.
    \end{equation}
\end{assumption}

\paragraph{Considered method.} A standard approach for solving problem \eqref{eq:problem} is the projected subgradient method, summarized in Algorithm~\ref{alg:proj}.

\begin{algorithm}[H]
    \caption{Projected subgradient method with generic stepsizes}
    \label{alg:proj}   
    \begin{algorithmic}[1]
        \REQUIRE number of iterations $N$, sequence of positive stepsizes $\{h_k\}_{1\leqslant k\leqslant N}$, convex set $X$, convex function $f$ defined on $X$, initial iterate $x^1\in X$
        \FOR{$k=1,\ldots, N$}
        \STATE Receive a subgradient $g^k \in \partial f(x^k)$
        \STATE Compute $x^{k+1} = P_X(x^k - h_kg^k)$
        \ENDFOR
        \ENSURE last iterate $x^{N+1}$
    \end{algorithmic}
\end{algorithm}

\noindent AdaGrad\footnote{The original AdaGrad \citep{duchi2011adaptive,mcmahan2010adaptive} uses coordinate-wise stepsizes, whereas the scalar variant in \eqref{h_k} is often called AdaGrad-Norm \citep{ward2020adagrad}. In this paper, we restrict attention to scalar stepsizes and refer to the instance of Algorithm~\ref{alg:proj} with stepsizes \eqref{h_k} simply as AdaGrad.} can be seen as a special case of Algorithm~\ref{alg:proj} with
\begin{equation}
    \label{h_k}
    h_k=\frac{h}{\sqrt{G^2+\sum\limits_{i=1}^{k}\|g^i\|^2}}\quad\text{for}\quad k=1,\ldots,N \text{ and some } h > 0.
\end{equation}
We assume that the starting point $x^1$ of the algorithm satisfies
\begin{equation}\label{neq_R}
    \|x^1-x^*\|\leqslant R.
\end{equation}
If $h > 0$ is fixed, then the last iterate of AdaGrad is not guaranteed to converge, as we illustrate in the following proposition.
\begin{proposition}\label{prop:bad_example}
    There exists problem \eqref{eq:problem} satisfying Assumptions~\ref{as:convexity} and \ref{as:bounded_subgrad} such that for any $N \geq 1$, $h > 0$ Algorithm~\ref{alg:proj} with stepsizes defined in \eqref{h_k} satisfies:
    \begin{equation}
        f(x^{N+1}) - f(x^*) \geq \frac{Gh}{\sqrt{2}}. \label{eq:lower_bound_1}
    \end{equation}
\end{proposition}
\begin{proof}
    Let $X = \R$ and $f(x) = G|x - x^1|$. Since subgradient $g^k$ can be selected arbitrarily and $\partial f(x^1) = [-G,G]$, consider the following choice: $g^k = 0$ for $k=1,\ldots, N-1$ and $g^{N} = G$. Then, we have: $x^k = x^1$ for $k=1,\ldots,N$ and
    \begin{equation*}
        x^{N+1} = x^1 - \frac{hG}{\sqrt{G^2 + G^2}} = x^1 - \frac{h}{\sqrt{2}} \Longrightarrow f(x^{N+1}) - f(x^*) = G|x^{N+1} - x^1| = \frac{Gh}{\sqrt{2}}.
    \end{equation*}
\end{proof}

Proposition~\ref{prop:bad_example} suggests that $h$ should either decrease with $k$ (e.g., as in \citet{liu2025online}) or depend explicitly on the horizon $N$. Therefore, we consider the choice $h = \nicefrac{R}{N^{\gamma}}$ for some $\gamma > 0$. Note that any fixed stepsize $h$ can be written as $h = \nicefrac{R}{N^{\gamma}}$ for a fixed natural $N$ and suitable $R > 0, \gamma > 0$.

\section{Related Works}\label{sec_related}

The literature on AdaGrad-type methods and, more broadly, first-order optimization methods is vast. We therefore focus on results most closely related to our setting: convex (possibly non-smooth) minimization of a Lipschitz function via AdaGrad-type methods, and the convergence of the \emph{last iterate} for the first-order methods.

\paragraph{Averaged/best-iterate guarantees for AdaGrad-type methods.} AdaGrad was introduced in online learning and stochastic optimization \citep{mcmahan2010adaptive,duchi2011adaptive}. The original analyses give regret bounds which, via standard online-to-batch conversion, translate into $\cO(\nicefrac{1}{\sqrt{N}})$ guarantees for the best/weighted-average iterate on convex Lipschitz objectives.
A large body of work extends these results in different directions (acceleration, parameter-free tuning, and richer preconditioners), see, e.g.,  \citep{gupta2017unified,mohri2016accelerating,levy2018online,orabona2015scale,cutkosky2018black,cutkosky2020parameter,cutkosky2019matrix,cutkosky2020better,cutkosky2020adaptive,wan2018efficient,feinberg2023sketchy,defazio2022momentumized}. Since our contribution concerns last-iterate guarantees, we do not survey this line in detail.

\paragraph{Last-iterate convergence for (stochastic) subgradient methods.} For general convex Lipschitz objectives, \citet{shamir2013stochastic} proved that projected SGD (and, as a special case, the deterministic projected subgradient method) with stepsizes $\eta_k \propto \nicefrac{1}{\sqrt{k}}$ achieves an expected last-iterate accuracy of order $\cO(\nicefrac{\log(N)}{\sqrt{N}})$ on bounded domains. \citet{harvey2019tight} showed that this logarithmic factor is unavoidable for this stepsize schedule by constructing matching lower bounds, and also derived high-probability analogues for SGD.
To obtain the information-theoretically optimal $\Theta(\nicefrac{1}{\sqrt{N}})$ last-iterate rate, \citet{jain2021making} proposed a horizon-dependent stepsize schedule (piecewise-constant, with a ``doubling'' structure\footnote{More precisely, \citet{jain2021making} propose the following stepsize schedule: $\eta_k \propto \nicefrac{2^{-i}}{\sqrt{N}}$, where $k \in (N_i, N_{i+1}]$ and $N_i$ are such that $0 =: N_0 < N_1 < \ldots < N_k = N-1 < N_{k+1} := N$, $k$ is the smallest integer such that $N\cdot 2^{-i} \leq 1$, and $N_i := N - \lceil N\cdot 2^{-i} \rceil$ for $i = 0,1, \ldots, k$.}) and proved both in-expectation and high-probability last-iterate bounds under bounded-domain assumptions.

\paragraph{Sharp deterministic rates and impossibility of anytime optimal schedules.}
\citet{zamani2025exact} developed a refined proof technique for the deterministic projected subgradient method based on tracking the distance to a \emph{time-varying} reference point. Using this approach, they obtained exact worst-case last-iterate bounds (with matching lower bounds including constants) for constant stepsizes and constant step lengths; optimizing the constant stepsize yields a rate of order $\cO(\nicefrac{\sqrt{\log(N)}}{\sqrt{N}})$. They further proposed a horizon-dependent ``optimal'' subgradient method with linearly decaying stepsizes $\eta_k = \nicefrac{R(N+1-k)}{G\sqrt{(N+1)^3}}$, achieving the minimax rate $\nicefrac{GR}{\sqrt{N}}$ and matching the lower bound of \citet{drori2016optimal}. They also showed that some dependence on the horizon is unavoidable: there is no universal stepsize sequence that achieves the optimal last-iterate bound at \emph{every} iteration.
\citet{kornowski2025gradient} strengthened this impossibility result by proving that any horizon-independent stepsize sequence must incur a polylogarithmic penalty, namely a lower bound of order $\Omega(\nicefrac{\log^{1/8}(N)}{\sqrt{N}})$.
Related work also studied other adaptive choices of stepsizes. In particular, \citet{zamani2024exact} analyzed Polyak stepsizes \citep{polyak1969minimization}: they proved that the standard Polyak rule $\eta_k = \nicefrac{(f(x^k) - f(x^*))}{\|g^k\|^2}$ yields a worst-case last-iterate rate $\cO(N^{-1/4})$ and constructed instances showing this is tight; they also introduced a modified, horizon-dependent Polyak rule $\eta_k = \nicefrac{(N+1-k)(f(x^k) - f(x^*))}{(N+1)\|g^k\|^2}$ which recovers the optimal $\cO(\nicefrac{1}{\sqrt{N}})$ rate.

\paragraph{From regret to last-iterate and implications for AdaGrad.}
A complementary viewpoint is to reduce last-iterate guarantees to online learning regret bounds. \citet{defazio2023optimal} established a general regret-to-last-iterate conversion which recovers (and simplifies) the analysis of horizon-dependent linear-decay schedules for subgradient methods, including the optimal schedule of \citet{zamani2025exact}.
This framework has recently been utilized by \citet{liu2025online} to obtain last-iterate guarantees in online convex optimization under heavy-tailed feedback for online gradient descent, dual averaging \citep{nesterov2009primal,xiao2009dual}, and AdaGrad. In particular, their AdaGrad last-iterate bound applies to a \emph{horizon-dependent, linearly decayed} variant of the stepsize (multiplying the standard AdaGrad step by $\nicefrac{(N-k)}{N}$), whereas our analysis studies the \emph{standard} AdaGrad stepsizes without a linearly decaying stepsize schedule.

Finally, \citet{liu2023revisiting} extended and refined the approach of \citet{zamani2025exact} to stochastic gradient methods, providing both in-expectation and high-probability last-iterate bounds under broader assumptions.

\section{Main Results}\label{sec_main} 

\paragraph{Key technical novelty.} A key technical ingredient of our analysis is an exponent parameter $\delta \geq 0$ defined as
    \begin{equation}
    \label{sim}
    \delta := \frac{1}{2}\log_N\left(1+\sum\limits_{i=1}^{N-1}\frac{\|g^i\|^2}{G^2}\right)\Longleftrightarrow G^2+\sum\limits_{i=1}^{N-1}\|g^i\|^2= G^2N^{2\delta}.
    \end{equation}
Under Assumption~\ref{as:bounded_subgrad}, we have $G^2\leqslant G^2+\sum_{i=1}^{N-1}\|g^i\|^2\leqslant G^2N$, so $\delta \in [0, \nicefrac{1}{2}]$. Since $\delta$ depends on the entire trajectory of the method, it is not known a priori and is used only as an analytical device. This parameterization allows us to derive tight upper bounds for terms that naturally arise in the proof and capture the interaction between the adaptive stepsizes $h_k$ and the subgradient norms $\|g^k\|$.

We now state our main upper bound on $f(x^{N+1}) - f(x^*)$.

\begin{theorem}
    \label{th_main}
    Let Assumptions~\ref{as:convexity} and \ref{as:bounded_subgrad} be satisfied. Assume that Algorithm~\ref{alg:proj} with stepsize \eqref{h_k} is run for $N$ iterations and $h = \nicefrac{R}{N^{\gamma}}$ for some $\gamma \in (0, \nicefrac{1}{2}]$. Then, the last iterate satisfies
    \begin{multline}
        \label{neq_Zamani_3}
        f(x^{N+1})-f(x^*)\leqslant 
        \frac{GR}{2}\Big(\frac{{N^{\gamma}}{\sqrt{N^{2\delta}+1}}}{2N+1}+
        \frac{N^{-\gamma}\left(4\log (N^{2\delta})+5\right)\sqrt{N^{2\delta}+1}}{2\max\{N^{2\delta}-1, 1\}}+\\N^{-\gamma-2\delta}\sqrt{N^{2\delta}+1}+N^{-\gamma-\delta}\Big), 
    \end{multline}
    where constant $\delta$ is defined in \eqref{sim}.
\end{theorem}
\begin{proof}
    The proof of this theorem is given in Appendix~\ref{sec_app_A}.
\end{proof}

Note that for all natural $N$, 
$$\sqrt{N^{2\delta}+1}\leqslant\sqrt2 N^\delta,\quad \dfrac{1}{2N+1}\leqslant\dfrac{1}{2N},\quad \frac{1}{\max\{N^{2\delta}-1,1\}}\leqslant2N^{-2\delta}.
$$
Therefore, the following corollary holds.
\begin{corollary} 
    Under the conditions of Theorem \ref{th_main}, the following bound holds:
    $$f(x^{N+1}) - f(x^*) \leqslant C_1 N^{\gamma+\delta-1} + C_2 N^{-\gamma-\delta} \log(N^{2\delta}) + C_3 N^{-\gamma-\delta},$$
    where 
    $C_1=\frac{\sqrt{2}}{4}GR, C_2=2\sqrt{2}GR, C_3=\frac{1}{2}GR(6\sqrt{2}+1).$
\end{corollary}
For $0<\gamma+\delta\leqslant\frac12$, the leading term is $C_2 N^{-\gamma-\delta} \log(N^{2\delta})$, whereas
for $\frac12<\gamma+\delta\leqslant1$, the dominant term is $C_1 N^{\gamma+\delta-1}$. This yields the following consequence.

\begin{corollary}\label{cor:main} 
    Under the conditions of Theorem \ref{th_main}, the following bound holds
    $$f(x^{N+1}) - f(x^*) =
        \begin{cases}
        O\left( GRN^{-\gamma-\delta} (1+\log(N^{2\delta}))\right),&0<\gamma+\delta\leqslant\frac12,\\
        O\left(GRN^{\gamma+\delta-1}\right),&\frac12<\gamma+\delta\leqslant1.
        \end{cases}
    $$
    In particular, for any choice of $\gamma \in (0,\nicefrac{1}{2}]$ there exists problem \eqref{eq:problem} satisfying Assumptions~\ref{as:convexity} and \ref{as:bounded_subgrad} such that \eqref{sim} holds with $\delta = \nicefrac{\lfloor \gamma + \nicefrac{3}{4} \rfloor}{2}$, implying that
    \begin{equation}
        f(x^{N+1}) - f(x^*) = O\left(\max\left\{\frac{GR}{N^{\gamma}}, \frac{GR}{N^{\frac{1}{2}- \gamma}}\right\}\right), \label{eq:worst_case_bound}
    \end{equation}
    where we used that $\delta = 0 \Longrightarrow \log(N^{2\delta}) = 0$ for $\gamma \in (0,\nicefrac{1}{4}]$. The upper bound above is minimized at $\gamma = \nicefrac{1}{4}$: in this case, $f(x^{N+1}) - f(x^*) = O(\nicefrac{GR}{N^{\nicefrac{1}{4}}})$.
\end{corollary}

Even for the best choice of $\gamma$, the bound in \eqref{eq:worst_case_bound} is worse than the standard $O(\nicefrac{GR}{N^{\nicefrac{1}{2}}})$ guarantee for the averaged iterate of AdaGrad \citep{duchi2011adaptive}. It is also worse than the $O(\nicefrac{GR\sqrt{\log(N)}}{N^{\nicefrac{1}{2}}})$ bound established for the last iterate of the (non-adaptive) subgradient method with a constant stepsize \citep{zamani2025exact}. These discrepancies naturally raise the question of whether \eqref{eq:worst_case_bound} is tight. The next theorem answers this question in the affirmative.

\begin{theorem}\label{thm:lower_bound}
    For any $\gamma > 0$ and natural $N$, there exists problem \eqref{eq:problem} with $d = 1$ for $\gamma \in (0,\nicefrac{1}{4}]$ and $d = \max\left\{\lceil 16 N^{1-2\gamma} \rceil, 1\right\}$ for $\gamma > \nicefrac{1}{4}$ satisfying Assumptions~\ref{as:convexity} and \ref{as:bounded_subgrad} such that after $N$ iterations of Algorithm~\ref{alg:proj} with stepsizes \eqref{h_k} and $h = \nicefrac{R}{N^\gamma}$, $R \geqslant \|x^1 - x^*\|$ the following inequality holds:
    \begin{equation}
        f(x^{N+1}) - f(x^*) = \Omega\left(\max\left\{\frac{GR}{N^{\gamma}}, \min\left\{\frac{GR}{N^{\frac{1}{2}- \gamma}}, GR\right\}\right\}\right).\label{eq:lower_bound}
    \end{equation}
\end{theorem}
\begin{proof}
    The proof of this theorem is given in Appendix~\ref{appendix:lower_bound}.
\end{proof}

Combining the upper bound in \eqref{eq:worst_case_bound} with the lower bound in \eqref{eq:lower_bound}, we conclude that the optimal worst-case last-iterate rate for AdaGrad with $h = \nicefrac{R}{N^\gamma}$ is $O(\nicefrac{GR}{N^{\nicefrac{1}{4}}})$, attained at $\gamma = \nicefrac{1}{4}$. We also note that for $\gamma > \nicefrac{1}{2}$ the lower bound becomes $\Omega(GR)$, which is why we do not consider $\gamma > \nicefrac{1}{2}$ in Theorem~\ref{th_main}.

\section{Numerical Experiments}\label{sec_exper}

\begin{figure}[t]
    \centering
    \includegraphics[width=1.0\linewidth]{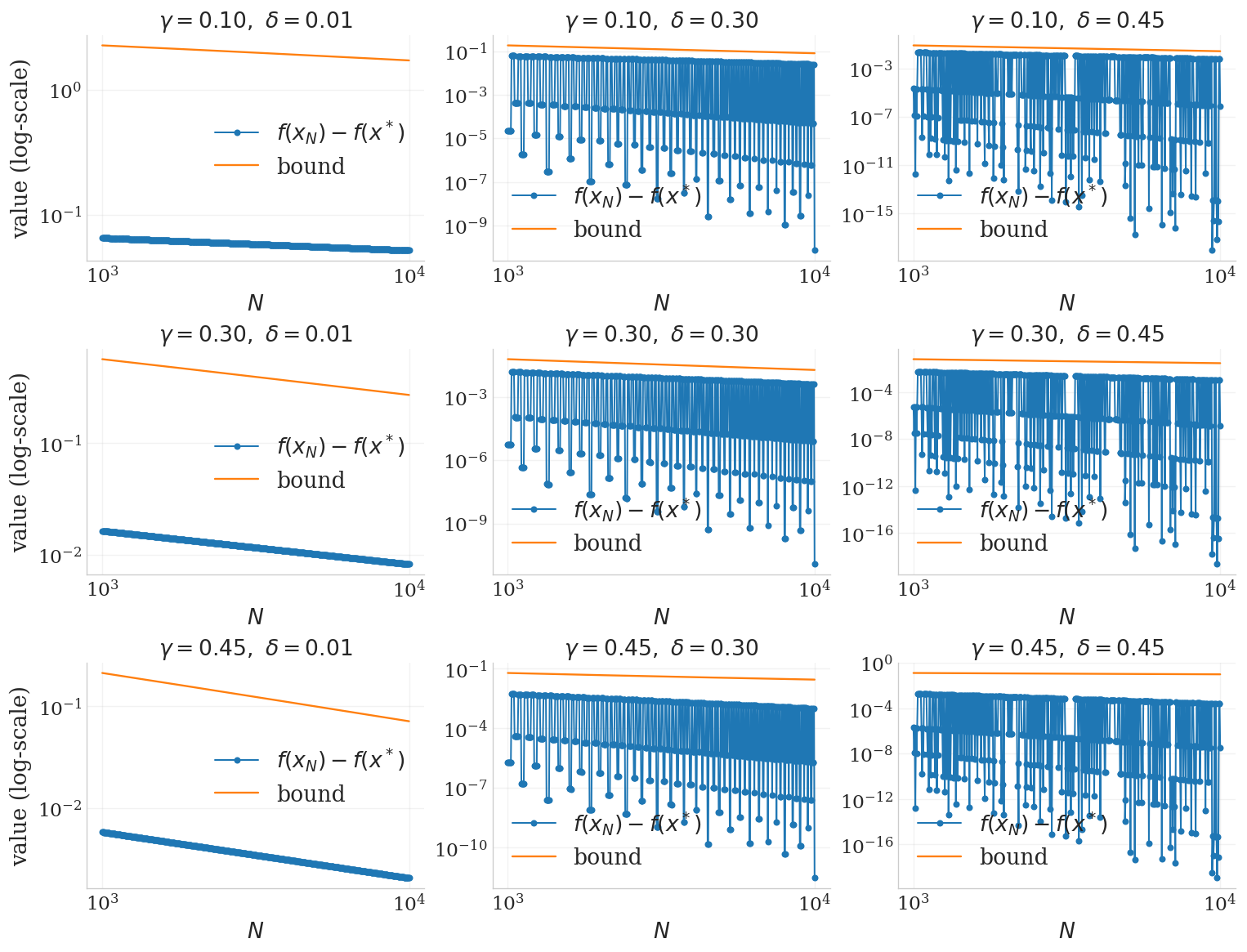}
    \caption{Empirical last-iterate error and the theoretical bound for AdaGrad with fixed values of \(\delta\) and different choices of \(\gamma\).}
    \label{fig:delta_fixed}
\end{figure}

\begin{figure}[t]
    \centering
    \includegraphics[width=1.0\linewidth]{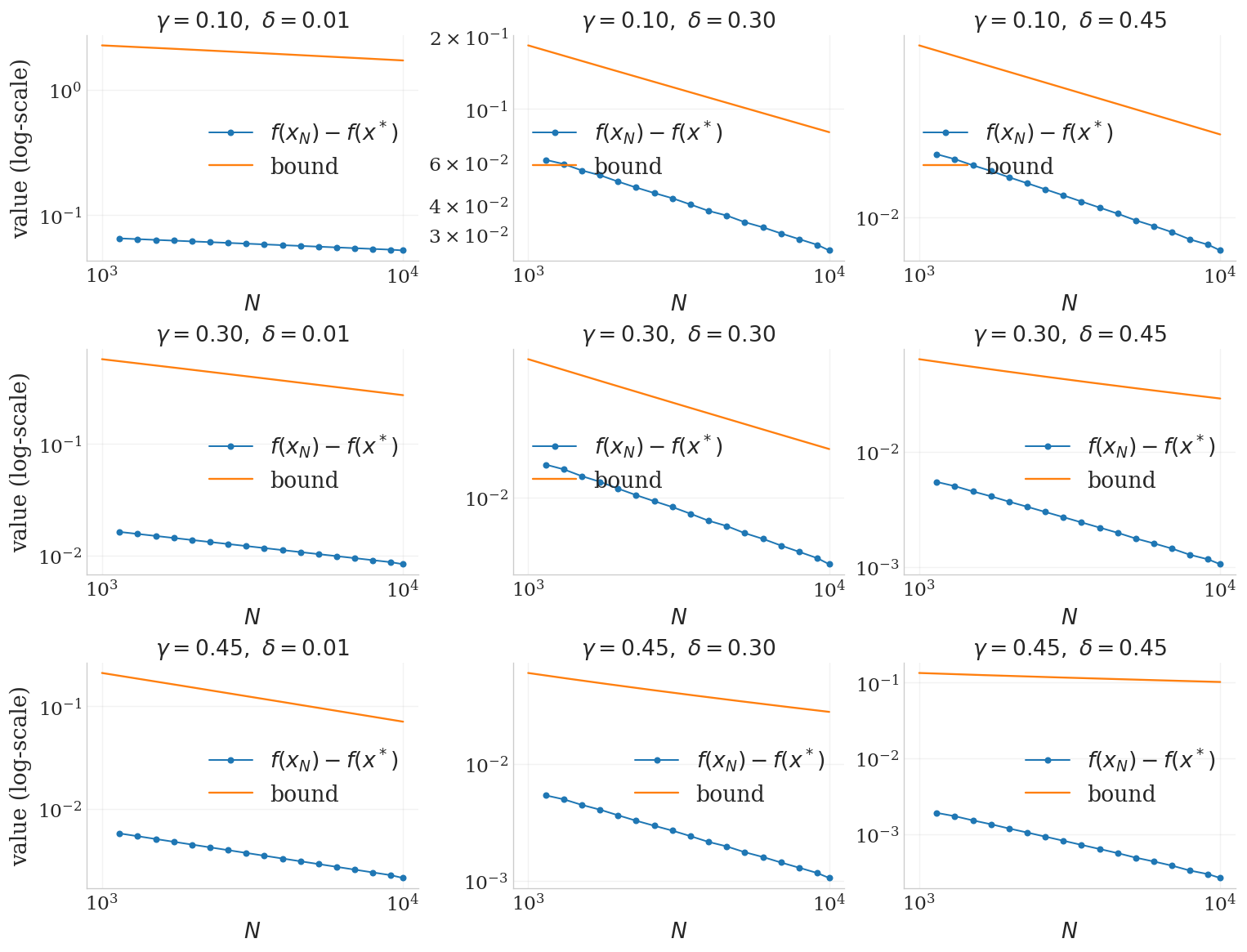}
    \caption{Empirical last-iterate error and the theoretical bound for AdaGrad with fixed values of \(\delta\) and different choices of \(\gamma\) with window maximum applied .}
    \label{fig:smooth_delta}
\end{figure}

\begin{table}[ht]
\centering
\caption{Empirical and theoretical slopes for different parameter combinations. Each cell shows the empirical slope $k_{\text{emp}}$ (top) and the bound slope $k_{\text{bound}}$ (bottom).}
\label{tab:slopes_comparison}
\begin{tabular}{c|S[table-format=-1.4] S[table-format=-1.4] S[table-format=-1.4] S[table-format=-1.4] S[table-format=-1.4] S[table-format=-1.4]}
\toprule
\multicolumn{1}{c|}{\multirow{2}{*}{$\delta$}} & \multicolumn{6}{c}{$\gamma$} \\
\cmidrule(lr){2-7}
 & {0.01} & {0.10} & {0.20} & {0.30} & {0.40} & {0.49} \\
\midrule
\multirow{2}{*}{0.01} & -0.0101 & -0.1006 & -0.2013 & -0.3019 & -0.4026 & -0.4931 \\
                       & -0.0305 & -0.1206 & -0.2209 & -0.3217 & -0.4233 & -0.5100 \\
\midrule
\multirow{2}{*}{0.10} & -0.5692 & -0.6606 & -0.7620 & -0.8635 & -0.9650 & -1.0564 \\
                       & -0.1294 & -0.2197 & -0.3209 & -0.4231 & -0.5190 & -0.5542 \\
\midrule
\multirow{2}{*}{0.20} & -0.1956 & -0.2856 & -0.3856 & -0.4857 & -0.5857 & -0.6757 \\
                       & -0.1913 & -0.2824 & -0.3850 & -0.4831 & -0.5250 & -0.4347 \\
\midrule
\multirow{2}{*}{0.30} & -0.3077 & -0.3981 & -0.4985 & -0.5989 & -0.6993 & -0.7897 \\
                       & -0.2677 & -0.3601 & -0.4596 & -0.5104 & -0.4164 & -0.2669 \\
\midrule
\multirow{2}{*}{0.40} & -0.4135 & -0.5045 & -0.6057 & -0.7069 & -0.8081 & -0.8991 \\
                       & -0.3565 & -0.4464 & -0.5038 & -0.4199 & -0.2540 & -0.1290 \\
\midrule
\multirow{2}{*}{0.49} & -0.5032 & -0.5939 & -0.6946 & -0.7954 & -0.8961 & -0.9868 \\
                       & -0.4388 & -0.4991 & -0.4392 & -0.2740 & -0.1302 & -0.0261 \\
\bottomrule
\end{tabular}
\end{table}

We conduct numerical experiments to assess how tight the bound of Theorem~\ref{th_main} is in practice for the last iterate of AdaGrad. We consider the one-dimensional convex non-smooth objective
$f(x) = |x|$,
with the initial point \( x^1 = 0 \). To construct a worst-case instance consistent with the theoretical assumptions, we define the following time-dependent subgradient oracle:
\[
g^k =
\begin{cases}
0, & \text{if } k < N - m , \\
-1, & \text{if } k > N - m \text{ and } x^k = 0, \\
\mathrm{sign}(x^k), & \text{if } k > N - m \text{ and } x^k \neq 0,
\end{cases}
\]
where \( m = \lceil N^{2\delta} \rceil \). This choice ensures that only the last \( m \) iterations contribute nontrivially to the cumulative squared subgradient norm, allowing us to explicitly control the parameter \( \delta \). 
For fixed values of the parameters \( \gamma \) and \( \delta \), we run Algorithm~\ref{alg:proj} for different values of \( N \). For each run, we compute the empirical error \( f(x^N) - f(x^\ast) \) and compare it with the theoretical upper bound from Theorem~\ref{th_main}.

As illustrated in Figure~\ref{fig:delta_fixed}, the dependence \( m = \lceil N^{2\delta} \rceil \) introduces visible discontinuities in the empirical curves. These ``jumps'' are caused by rounding in the definition of \( m \): the effective value of \( \delta \) realized in a given run does not exactly coincide with the prescribed one, but only approximates it.

To reduce the visual impact of the discontinuities caused by the rounding of $m = \lceil N^{2\delta}\rceil$, we smooth the empirical curves by applying a windowed maximum: for each consecutive block of iterations we plot the maximum value observed in that block. This operation preserves the worst‑case trend of the sequence and allows us to approximate the resulting curve by a linear function and compare the slope in log‑log scale. The curves smoothed with the windowed maximum are shown in Figure ~\ref{fig:smooth_delta}, and the results of this comparison are presented in Table~\ref{tab:slopes_comparison}.

\section{Conclusion}\label{ref_concl}

In this paper, we derive the convergence bound $O\left(\max\left\{\nicefrac{GR}{N^{\gamma}}, \nicefrac{GR}{N^{\frac{1}{2}- \gamma}}\right\}\right)$ for the last iterate of AdaGrad with the stepsize parameter $h = \nicefrac{R}{N^\gamma}$, $\gamma \in (0,\nicefrac{1}{2}]$, in convex optimization with $G$-bounded subgradients, and we show that this guarantee is tight by providing a matching lower bound. In particular, our results imply that the minimax-optimal convergence rate attained by the last iterate of standard AdaGrad is $O(\nicefrac{GR}{N^{\nicefrac{1}{4}}})$. Our main technical contribution is the introduction of the parameter $\delta$, which describes the growth of the cumulative squared subgradients. This viewpoint enables tight upper bounds for the sums that naturally arise in the analysis and capture the non-trivial interaction between the adaptive stepsizes $h_k$ and the squared subgradient norms $\|g^k\|^2$. We expect this approach to be useful for analyzing other AdaGrad-like methods.

Our paper also leaves several interesting open problems. In particular, it would be valuable to extend our analysis to the classical coordinate-wise version of AdaGrad and to stochastic settings. Another promising direction is to study last-iterate convergence of AdaGrad for smooth convex objectives. Finally, it would be interesting to investigate the last-iterate convergence of related methods such as RMSProp \citep{hinton2012neural} and Adam \citep{kingma2015adam}.

\section*{Acknowledgments}
The work of Margarita Preobrazhenskaia, Makar
Sidorov, and Igor Preobrazhenskii was carried out within the framework of a development programme for the Regional Scientific and Educational Mathematical Center of the Yaroslavl State University with financial support from the Ministry of Science and Higher Education of the Russian Federation (Agreement on provision of subsidy from the federal budget No. 075-02-2026-1331).

\bibliographystyle{apalike}
\bibliography{./refs}

\appendix

\section{Proof of Theorem~\ref{th_main}}\label{sec_app_A}

\subsection{Proof Sketch}

We sketch the proof of Theorem~\ref{th_main}. Our analysis follows an asymptotic viewpoint and builds on the following inequality for the last iterate of subgradient methods, proved by \citet{zamani2025exact}.

\begin{lemma} [\citet{zamani2025exact}]
\label{lem_Zamani}
Let $f$ be a convex function and $X$ be a closed convex set. Let $\hat{x}\in X$, $h_{N+1}>0$, and
\begin{equation}
\label{cond_v}
0\leqslant v_0\leqslant v_1\leqslant\ldots\leqslant v_N\leqslant v_{N+1}.
\end{equation}
Then, if the subgradient method (Algorithm 1) with starting point $x^1\in X$ generates $\{(x^k,g^k)\}$, then the inequality holds:
\begin{equation}
\label{neq_Zamani}
\sum_{k=1}^{N+1}c_k(f(x^k)-f(\hat{x}))\leqslant\frac{v_0^2}{2}\|x^1-\hat{x}\|^2+\frac12\sum_{k=1}^{N+1}h_k^2v_k^2\|g^k\|^2,
\end{equation}
\begin{equation}
\label{ck}
\text{where}\quad c_k=h_kv_k^2-(v_k-v_{k-1})\sum_{i=k}^{N+1}h_iv_i, \quad k=1,\ldots,N+1.
\end{equation}
\end{lemma}

We choose the numbers $v_k$ so that all coefficients in~(\ref{ck}) vanish except the last one:
$$
v_{N+1}=1,\quad v_N=1, \quad v_k=\frac{h_{N+1}}{\sum\limits_{i=k+1}^{N+1}h_iv_i}\quad\text{for}\quad k=0,\ldots,N.
$$
Lemma~\ref{lem_ck01} in Appendix~\ref{sec_tech} shows that this choice of $v_k$ guarantees $c_{N+1}=h_{N+1}$ and $c_k=0$ for $k=1,\ldots,N.$
Consequently, for the above $v_k$, inequality~(\ref{neq_Zamani}) reduces to
for chosen $v_k$, inequality (\ref{neq_Zamani}) takes the form
\begin{equation*}
f(x^{N+1})-f(\hat{x})\leqslant\frac{v_0^2}{2h_{N+1}}\|x^1-\hat{x}\|^2+\frac{1}{2h_{N+1}}\sum_{k=1}^{N+1}h_k^2v_k^2\|g^k\|^2.
\end{equation*}
Next, we separately evaluate the $1$-st, $N$-th, $(N+1)$-th terms on the right-hand side, as well as the sum
\begin{equation}
\label{sum_N-1}
    \sum\limits_{k=1}^{N-1}h_k^2v_k^2\|g^k\|^2.
\end{equation}

Lemma~\ref{lem_x} is the key tool for estimating \eqref{sum_N-1}. The basic idea is motivated by asymptotic analysis and consists of treating the sum $\sum_{k=1}^{N-1}h_k^2v_k^2\|g^k\|^2$ as a quantity of order $N^{2\delta}$. This approach is standard for identifying leading orders of growth. However, in our case, we do not assume that $N$ is large and instead derive estimates that hold for any finite $N$. Lemma~\ref{lem_x} is proved in Appendix~\ref{sec_key_lem}.

Appendix~\ref{sec_tech} is devoted to technical results. It establishes properties of the sequence $v_k$, similar to those in \citet{zamani2025exact}, and proves Lemma~\ref{lem_ck01}, which is used to bound the first and the $(N+1)$-th terms in the sum~(\ref{sum_N-1}). The final reasoning that completes the proof of Theorem~\ref{th_main} is presented in Appendix~\ref{sec_end_proof}.

\subsection{Key Lemma}\label{sec_key_lem}

\begin{lemma}
\label{lem_x}
Let $y_i$, $i=1,\ldots,N-1,$ be such that $0\leqslant y_i \leqslant 1$, and $1+\sum_{i=1}^{N-1}y_i= N^{2\delta}$, $\delta\in[0,\frac12]$. Then,
\begin{equation}
    \sum_{k=1}^{N-1}\frac{y_k}{(N-k)(1+\sum_{i=1}^ky_i)}\leqslant \frac{4\log (N^{2\delta})+5}{\max\{N^{2\delta}-1, 1\}}. \label{eq:main_technical_ineq}
\end{equation}
\end{lemma}
\begin{proof}
We introduce the following notation: $$S_k:=1+\sum\limits_{i=1}^k y_i \quad\text{and}\quad m:=\left\lceil N^{2 \delta}\right\rceil.$$
The latter means that $m-1<N^{2\delta}\leqslant m$. We consider three possible cases.

\paragraph{Case $1$: $m \geqslant 3$.} We will show that for $m \geqslant 3$
\begin{equation}
\label{neq_star}
\sum_{k=1}^{N-1} \frac{y_k}{(N-k)S_k} \leqslant \frac{4 H_{m-2}+1}{m-1} \leqslant \frac{ 4\log (N^{2\delta})+5}{N^{2 \delta}-1},
\end{equation}
where $H_k=1+\frac{1}{2}+\cdots+\frac{1}{k}$ is the $k$-th harmonic number.

We split the indices into ``layers'' based on the level of partial sums. For $j=2, \ldots, m-1$, we set
$$
\mathcal{I}_j:=\left\{k:j<S_k \leqslant j+1\right\} \quad \text{and} \quad \mathcal{I}_1:=\left\{k:1\leqslant S_k \leqslant 2\right\}.
$$
At each layer, the following holds: $S_k \geqslant j$ for $k \in \mathcal{I}_j$, and $\sum_{k \in \mathcal{I}_j} y_k \leqslant 2$.
Therefore, for $j=1,\ldots,m-1$
\begin{equation} 
\label{neq_Ij} 
\sum_{k \in \mathcal{I}_j} \frac{y_k}{(N-k)S_k} \leqslant \frac{1}{j} \sum_{k \in \mathcal{I}_j} \frac{y_k}{N-k} \leqslant \frac{1}{j} \max _{k \in \mathcal{I}_j} \frac{2}{N-k}.
\end{equation}
Let $k_j:=\max_{k \in \mathcal{I}_j} k$. From $S_{k_j}\leqslant j+1$ and $N^{2 \delta}=S_{N-1} \leqslant S_{k_j} + \sum_{i=k_j+1}^{N-1}y_i \leqslant S_{k_j} +  N-k_j-1$, we obtain that for any $k\in\mathcal{I}_j$
and $j=1,\ldots,m-2$
$$
N-k\geqslant N-k_j\geqslant N^{2 \delta}+1-S_{k_j} \geqslant\lceil N^{2 \delta}\rceil-j-1=m-j-1.$$
For $j=m-1$, we have $k_{m-1}:=\max_{k\in \mathcal{I}_{m-1}} k = N-1$, i.e., for $k \in \mathcal{I}_{m-1}$ we have $N-k \geq N-k_{m-1}=1$. Substituting into (\ref{neq_Ij}) and summing over layers, we obtain
\begin{multline*}
\sum_{k=1}^{N-1} \frac{y_k}{(N-k) S_k} \leqslant \sum_{j=1}^{m-2} \frac{2}{j(m-j-1)}+\frac{1}{m-1}=
\\\frac{2}{m-1}\sum_{j=1}^{m-2} \left(\frac{1}{j}+\frac{1}{m-j-1}\right)+\frac{1}{m-1}=
\frac{1}{m-1} \left(4\sum_{j=1}^{m-2} \frac{1}{j}+1\right)=
\frac{4 H_{m-2}+1}{m-1},
\end{multline*}
which yields the first inequality in (\ref{neq_star}). Next, we use $H_{m-2} < \log (m-2)+1 \leqslant \log \left(N^{2 \delta}-1\right)+1 $, as well as $m-1 \geqslant N^{2 \delta}-1$, yielding the second part of (\ref{neq_star}).

\paragraph{Case $2$: $m = 2$.} In this case, we have that $\mathcal{I}_1 = \{1,\ldots, N-1\}$ and
\begin{equation*}
    \sum_{k=1}^{N-1} \frac{y_k}{(N-k) S_k} = \sum_{k\in \mathcal{I}_1} \frac{y_k}{(N-k) S_k} \overset{\eqref{neq_Ij}}{\leqslant} 2 \leqslant \frac{4\log (N^{2\delta})+5}{\max\{N^{2\delta}-1, 1\}},
\end{equation*}
since $\max\{N^{2\delta}-1, 1\} = 1$ for $m = 2$.

\paragraph{Case $3$: $m = 1$.} If $m=1$, then $\delta = 0$ and $y_i = 0$ for all $i = 1,\ldots, N-1$. In this case, \eqref{eq:main_technical_ineq} is satisfied since the left-hand side of \eqref{eq:main_technical_ineq} equals $0$, while the right-hand side of \eqref{eq:main_technical_ineq} equals $5$.

\end{proof}

\subsection{Technical Results}\label{sec_tech}
\begin{lemma}
\label{lem_ck01}
Let
\begin{equation}
\label{v_choise}
v_{N+1}=1,\quad v_k=\frac{h_{N+1}}{\sum\limits_{i=k+1}^{N+1}h_iv_i}\quad\text{for}\quad k=0,\ldots,N.
\end{equation} Then $v_N=1$, and for $c_k$ defined by equality (\ref{ck}), the following is true:
\begin{equation}
\label{ck01}
c_{N+1}=h_{N+1},\quad c_k=0\quad\text{for}\quad k=1,\ldots,N.
\end{equation}
\end{lemma}
\begin{proof}
The expression for $v_k$, according to (\ref{v_choise}), for $k=N$:
$$v_N=\frac{h_{N+1}}{h_{N+1}v_{N+1}}=1.$$
The expression for $c_k$, according to (\ref{ck}), for $k=N+1$:
$$c_{N+1}=h_{N+1}v_{N+1}^2-(v_{N+1}-v_{N})h_{N+1}v_{N+1}=v_Nh_{N+1}v_{N+1}=h_{N+1}.$$
The expression for $c_k$, according to (\ref{ck}), with $k=1,\ldots,N$: 
\begin{multline*} 
c_k=h_kv_k^2-(v_k-v_{k-1})\sum_{i=k}^{N+1}h_iv_i=\\ 
h_kv_k^2-v_k(v_kh_k+\sum_{i=k+1}^{N+1}h_iv_i)+v_{k-1}\sum_{i=k}^{N+1}h_iv_i=\\ 
-v_k\sum_{i=k+1}^{N+1}h_iv_i+v_{k-1}\sum_{i=k}^{N+1}h_iv_i=h_{N+1}-h_{N+1}=0. 
\end{multline*}
\end{proof}

\begin{remark}
If $h_i>0$, then $v_k$, defined by the equalities (\ref{v_choise}), are positive and increasing, i.e., (\ref{cond_v}) holds.
\end{remark}

\begin{remark}
For $v_k$ chosen according to (\ref{v_choise}), inequality (\ref{neq_Zamani}) takes the form
\begin{equation}
\label{neq_Zamani_1}
f(x^{N+1})-f(\hat{x})\leqslant\frac{v_0^2}{2h_{N+1}}\|x^1-\hat{x}\|^2+\frac{1}{2h_{N+1}}\sum_{k=1}^{N+1}h_k^2v_k^2\|g^k\|^2.
\end{equation}
\end{remark}

\begin{lemma}
\label{lem_vk_prop}
The coefficients $v_k$, defined in (\ref{v_choise}), have the following properties.
\begin{enumerate} 
\item $\dfrac{1}{v_{k}}=\dfrac{1}{v_{k+1}}+\dfrac{h_{k+1} v_{k+1}}{h_{N+1}}$ for $k=0,\ldots,N-1$. 
\item $\dfrac{1}{v_k^2}=1+\dfrac{2}{h_{N+1}}\sum\limits_{i=k+1}^{N}h_i+\dfrac{1}{h_{N+1}^2}\sum\limits_{i=k+1}^{N}h_i^2v_i^2$  for $k=0,\ldots,N-1$.
\end{enumerate}
\end{lemma}
\begin{proof} 
\begin{enumerate} 
\item This part follows from $\dfrac{1}{v_k}=\dfrac{1}{h_{N+1}}\sum\limits_{i=k+1}^{N+1}h_iv_i = \dfrac{1}{h_{N+1}}\sum\limits_{i=k+2}^{N+1}h_iv_i + \dfrac{h_{k+1}v_{k+1}}{h_{N+1}} = \dfrac{1}{v_{k+1}}+\dfrac{h_{k+1} v_{k+1}}{h_{N+1}}$ for $k=0,\ldots,N-1$.

\item From part 1 of this lemma we have $\dfrac{1}{v_k^2}=\left(\dfrac{1}{v_{k+1}}+\dfrac{h_{k+1} v_{k+1}}{h_{N+1}}\right)^2=\dfrac{1}
{v_{k+1}^2}+\dfrac{2h_{k+1}}{h_{N+1}}+\dfrac{h_{k+1}^2 v_{k+1}^2}{h_{N+1}^2}=\dfrac{1}
{v_{N+1}^2}+\dfrac{2}{h_{N+1}}\sum\limits_{i=k+1}^{N}h_i+\dfrac{1}{h_{N+1}^2}\sum\limits_{i=k+1}^{N}h_i^2v_i^2=1+\dfrac{2}{h_{N+1}}\sum\limits_{i=k+1}^{N}h_i+\dfrac{1}{h_{N+1}^2}\sum\limits_{i=k+1}^{N}h_i^2v_i^2.$ 
\end{enumerate}
\end{proof}

\begin{corollary}
\label{col_vk_neq}
For $k=0,\ldots,N-1$, the inequality
\begin{equation}
\label{neq_vk}
v_{k}^2\leqslant
\dfrac{1}{2(N-k)+1}.
\end{equation}
\end{corollary}
\begin{proof} 
Using part 2 of Lemma \ref{lem_vk_prop} 
and the fact that $h_j\geqslant h_{N+1}$ for $j=1,\ldots,N+1$, 
we get 
$$\dfrac{1}{v_k^2}\geqslant 1+\dfrac{2}{h_{N+1}}\sum\limits_{i=k+1}^{N}h_i\geqslant1+\dfrac{2}{h_{N+1}}\sum\limits_{i=k+1}^{N}h_{N+1}=1+2(N-k)$$
for $k=0,\ldots,N-1$. Then we have (\ref{neq_vk}).
\end{proof}

\begin{remark}\label{remark:h_N+1}
    Note that this is the first time in the proof, when we used some assumptions on sequence $h_j$ (besides $h_j > 0$), i.e., that $h_j\geqslant h_{N+1}$ for $j=1,\ldots,N+1$. Since $h_{N+1}$ is not used to generate $x^{N+1}$, we can select $h_{N+1} := h_N$ to satisfy all the conditions needed for the proof.
\end{remark}

\subsection{Proof of Theorem \ref{th_main}}\label{sec_end_proof}
From (\ref{neq_Zamani_1}), using $v_{N+1}=v_N=1$ and (\ref{neq_G}), we obtain
\begin{equation}
\label{neq_Zamani_2}
f(x^{N+1})-f(\hat{x})\leqslant\frac{R^2v_0^2}{2h_{N+1}}+\frac{1}{2h_{N+1}}\sum_{k=1}^{N-1}{h_k^2v_k^2\|g^{k}\|^2}+
\\\frac{G^2h_N^2}{2h_{N+1}}+\dfrac{G^2}2h_{N+1}.
\end{equation}
Due to (\ref{h_k}), (\ref{sim}), and Remark~\ref{remark:h_N+1} ($h_{N+1} := h_N$), we have
\begin{equation} 
\label{h_sim1} 
\frac{RN^{-\gamma}}{G\sqrt{N^{2\delta}+1}}\leqslant h_N
\leqslant\frac RGN^{-\gamma-\delta},
\end{equation}
\begin{equation} 
\label{h_sim2}  
\frac{RN^{-\gamma}}{G\sqrt{N^{2\delta}+1}}\leqslant h_{N+1}
\leqslant\frac{R}{G}N^{-\gamma-\delta}.
\end{equation}

To prove (\ref{neq_Zamani_3}), we upper bound each term from the right-hand side of inequality~(\ref{neq_Zamani_2}) separately.

\begin{enumerate}

\item Using Corollary \ref{col_vk_neq} for $k=0$ and inequality~(\ref{h_sim1}), we obtain that $v_0^2\leqslant\dfrac{1}{2N+1}$ and the first term is bounded as
\begin{eqnarray}
    \frac{R^2v_0^2}{2h_{N+1}} &\leqslant& \frac{R^2}{2h_{N+1}(2N+1)}  \leqslant \frac{R^2}{2(2N+1)}\frac{G\sqrt{N^{2\delta}+1}}{RN^{-\gamma}}\notag \\
    &=& 
\frac{RG}{2}\frac{{N^{\gamma}}{\sqrt{N^{2\delta}+1}}}{2N+1}.\label{neq1}
\end{eqnarray}

\item Now we estimate the sum $\frac{1}{2h_{N+1}}\sum_{k=1}^{N-1}{h_k^2v_k^2\|g^{k}\|^2}$. Using Corollary \ref{col_vk_neq}  
and formula (\ref{h_k}), we have $v_{k}^2\leqslant\dfrac{1}{2(N-k)+1}$ and obtain
\begin{multline*} 
\frac{1}{2h_{N+1}}\sum_{k=1}^{N-1}{h_k^2v^2_k\|g^{k}\|^2}\leqslant\dfrac{1}{2h_{N+1}}\sum\limits_{k=1 }^{N-1}\frac{h_k^2\|g^k\|^2}{2(N-k)+1}\leqslant\dfrac{1}{4h_{N+1}}\sum\limits_{k=1}^{N-1}\frac{h_k^2\|g^k\|^2}{N-k}= 
\\\dfrac{h^2}{4h_{N+1}}\sum\limits_{k=1}^{N-1}\frac{\|g^k\|^2}{(N-k)\left(G^2+\sum\limits_{i=1}^{k}\|g^i\|^2\right)}.
\end{multline*}
The multiplier in front of the sum is upper-bounded as:
$$ 
\dfrac{h^2}{4h_{N+1}} \overset{\eqref{h_sim2}}{\leqslant} \dfrac{R^2N^{-2\gamma}}{4}\frac{G\sqrt{N^{2\delta}+1}}{RN^{-\gamma}}=\frac{RG}{4}N^{-\gamma}\sqrt{N^{2\delta}+1}.
$$
To estimate the sum $\sum_{k=1}^{N-1}\frac{\|g^k\|^2}{(N-k)(G^2+\sum_{i=1}^{k}\|g^i\|^2)}$, we apply Lemma \ref{lem_x} for $y_k=\dfrac{\|g^k\|^2}{G^2}$, $k=1,\ldots,N-1$.
Then
\begin{eqnarray*}
    \sum\limits_{k=1}^{N-1}\dfrac{\|g^k\|^2}{(N-k)\left(G^2+\sum\limits_{i=1}^{k}\|g^i\|^2\right)} &=& \sum\limits_{k=1}^{N-1}\dfrac{\nicefrac{\|g^k\|^2}{G^2}}{(N-k)\left(1+\sum\limits_{i=1}^{k}\nicefrac{\|g^i\|^2}{G^2}\right)}\\
    &\leqslant&  
    \frac{4\log (N^{2\delta})+5}{\max\{N^{2\delta}-1, 1\}}.
\end{eqnarray*}
Hence,
\begin{equation} 
\label{neq2}
\dfrac{1}{2h_{N+1}}\sum\limits_{k=1}^{N-1}{h_k^2v_k^2\|g^{k}\|^2}\leqslant\frac{RG}{4}N^{-\gamma}\frac{\sqrt{N^{2\delta}+1}}{\max\{N^{2\delta}-1, 1\}}(4\log (N^{2\delta})+5).
\end{equation}

\item Taking into account inequalities (\ref{h_sim1}) and (\ref{h_sim2}), for the $3$-rd term in the right-hand side of (\ref{neq_Zamani_2}) we obtain
\begin{equation}
\label{neq3}
\frac{G^2h_N^2}{2h_{N+1}}\leqslant\frac{G^2}{2}\frac {R^2}{G^2}N^{-2\gamma-2\delta}\frac{G\sqrt{N^{2\delta}+1}}{RN^{-\gamma}}=\frac{RG}{2}N^{-\gamma-2\delta}\sqrt{N^{2\delta}+1}.
\end{equation}

\item Given inequality (\ref{h_sim2}), for the last term in the right-hand side of (\ref{neq_Zamani_2}) we have
\begin{equation}
\label{neq4}
\dfrac{G^2}2h_{N+1}\leqslant\frac{G^2}{2}\frac{R}{G}N^{-\gamma-\delta}=\frac{RG}{2}N^{-\gamma-\delta}.
\end{equation}

\end{enumerate}

Summing inequalities (\ref{neq1})--(\ref{neq4}) and taking $\hat x = x^*$, we obtain (\ref{neq_Zamani_3}).

\section{Proof of Theorem~\ref{thm:lower_bound}}\label{appendix:lower_bound}

Recall from Corollary~\ref{cor:main} that
\begin{equation*}
f(x^{N+1}) - f(x^*)
= O\left(
\max\left\{
\frac{GR}{N^\gamma},
\frac{GR}{N^{\frac12-\gamma}}
\right\}
\right),
\end{equation*}
We prove that this upper bound is tight by considering two possible situations.

\paragraph{Case 1: $\gamma \in (0,\frac14]$.} Let $X=\R$ and consider the problem
\begin{equation*}
\min_{x\in\R} f(x),
\quad
f(x)=G|x-x_1|,
\end{equation*}
for given $x^1 \in \R$. Then $x^*=x^1$, and Assumptions~\ref{as:convexity} and \ref{as:bounded_subgrad} hold. Since $g^k$ may be chosen arbitrarily from $\partial f(x^k)$ and $\partial f(x^1)=[-G,G]$, we choose
\begin{equation*}
g^k=0 \quad \text{for } k=1,\dots,N-1,
\quad
g^N=G.
\end{equation*}
Then $x^k=x^1$ for $k=1,\dots,N$, and therefore
\begin{equation*}
G^2+\sum_{i=1}^{N}\|g^i\|^2
= G^2+G^2
= 2G^2 \Longrightarrow h_N = \frac{h}{\sqrt2\,G}.
\end{equation*}
The last iterate satisfies
\begin{equation*}
x^{N+1}
= x^1 - h_N G
= x^1 - \frac{h}{\sqrt2}.
\end{equation*}
Since $h= \nicefrac{R}{N^\gamma}$ for some $R > 0$, we obtain
\begin{equation*}
f(x^{N+1})-f(x^\ast)
= G|x^{N+1}-x^1|
= \frac{Gh}{\sqrt2}
= \frac{GR}{\sqrt2\,N^\gamma} = \Omega\left(\frac{GR}{N^\gamma}\right).
\end{equation*}

\paragraph{Case 2: $\gamma  > \frac{1}{4}$.} Let $X=\R^d$ and define
\begin{equation}
f(x)=G\|x\|_\infty
= G\max_{1\le i\le d}|x_i|.
\end{equation}
Then $x^\ast=0$, and Assumptions~\ref{as:convexity} and \ref{as:bounded_subgrad} hold. Without loss of generality, we choose the initial point
\begin{equation*}
x^1=\left(\frac{R}{\sqrt d},\dots,\frac{R}{\sqrt d}\right)^\top,
\quad
\|x^1\| =R.
\end{equation*}
Indeed, for arbitrary initialization $x^1$, one can always change the coordinates in such a way that $x^1$ is defined as written above. Note that $\partial f(x)
= G\,\mathrm{conv}\{e_i \operatorname{sign}(x_i): i\in I(x)\}$, where $e_i$ is the $i$-th basis vector and $I(x)=\{i:\ |x_i|=\|x\|_\infty\}$. At each iteration, we choose
\begin{equation*}
g^k = Ge_{i_k} \operatorname{sign}(x_{i_k}^k) \quad \text{for some}\quad i_k \in I(x^k).
\end{equation*}
In this case, only the $i_k$-th coordinate is updated in iterations $k$ and $\|g^k\|_1 = \|g^k\|=G$ for all $k = 1,\ldots, N$. After $N$ iterations,
\begin{equation*}
x^{N+1}
= x^1 - \sum_{k=1}^N h_k g^k,
\end{equation*}
hence, the total $\ell_1$-distance ``traveled'' by the method is bounded as
\begin{equation*}
\|x^{N+1}-x^1\|_1
\leqslant 
\sum_{k=1}^N h_k\|g^k\|_1
= G\sum_{k=1}^N h_k.
\end{equation*}
Using
\begin{equation*}
h_k
= \frac{h}{\sqrt{G^2+\sum_{i=1}^k\|g^i\|^2}}
= \frac{h}{G\sqrt{1+k}},
\end{equation*}
we obtain
\begin{equation*}
\sum_{k=1}^N h_k
= \frac{h}{G}\sum_{k=1}^N\frac1{\sqrt{1+k}}
\leqslant \frac{h}{G}\int_1^{N+1} \frac{dx}{\sqrt x}
\leqslant \frac{2h}{G}\sqrt N.
\end{equation*}
Therefore, $\|x^{N+1}-x^1\|_1
\leqslant 2h\sqrt N
= 2R N^{\frac12-\gamma}$. Since $\|x^{N+1}-x^1\|_1 = \sum_{i=1}^d|x_{i}^{N+1}-x_{i}^1|$, there exists
$i\in\{1,\dots,d\}$ such that
\begin{equation*}
|x_{i}^{N+1}-x_{i}^1|
\leqslant \frac{2R N^{\frac12-\gamma}}{d}.
\end{equation*}
Choosing $d = \max\left\{\lceil 16 N^{1-2\gamma} \rceil, 1\right\}$, we ensure that
\begin{equation*}
|x_{i}^{N+1}-x_{i}^1|
\leqslant \frac{2R N^{\frac12-\gamma}}{d}
\leqslant \frac{R}{2\sqrt d}.
\end{equation*}
Hence,
\begin{equation*}
x_{i}^{N+1}
\geqslant x_{i}^1 - |x_{i}^{N+1}-x_{i}^1| \geqslant \frac{R}{\sqrt d}-\frac{R}{2\sqrt d}
= \frac{R}{2\sqrt d}.
\end{equation*}
Therefore, taking into account that $d = \max\left\{\lceil 16 N^{1-2\gamma} \rceil, 1\right\}$, we get
\begin{eqnarray*}
    f(x^{N+1})-f(x^*)
&=& G\|x^{N+1}\|_\infty
\geqslant \frac{GR}{2\sqrt d} = \frac{GR}{2\sqrt{\max\left\{\lceil 16 N^{1-2\gamma} \rceil, 1\right\}}}\\
&=& \Omega\left(
\min\left\{\frac{GR}{N^{\frac12-\gamma}}, GR\right\}
\right).
\end{eqnarray*}

\paragraph{Final bound.} Combining Cases 1 and 2, we conclude that for any
$\gamma > 0$,
\begin{equation*}
f(x^{N+1})-f(x^*)
= \Omega\left(
\max\left\{
\frac{GR}{N^\gamma},
\min\left\{\frac{GR}{N^{\frac12-\gamma}}, GR\right\}
\right\}
\right).
\end{equation*}

\end{document}